\documentclass{article}
\usepackage{amsmath}
\usepackage{amssymb}
\usepackage{mathtools}
\usepackage{hyperref}
\usepackage{amsthm}

\newtheorem{theorem}{Theorem}[section]
\newtheorem{corollary}[theorem]{Corollary}

\newtheorem{lemma}[theorem]{Lemma}
\newtheorem{defn}[theorem]{Definition}

\newcommand{\zee}{\mathbb{Z}}
\newcommand{\arr}{\mathbb{R}}
\newcommand{\eff}{\mathbb{F}}
\newcommand{\enn}{\mathbb{N}}
\newcommand{\iso}{\cong}
\title{Embeddings of weighted graphs in Erd\H{o}s-type settings}
\author{David M. Soukup}
\date{\today}

\begin{document}
\maketitle

\begin{abstract}
Many recent results in combinatorics concern the relationship between the size of a set and the number of distances determined by pairs of points in the set. One extension of this question considers configurations within the set with a specified pattern of distances. In this paper, we use graph-theoretic methods to prove that a sufficiently large set $E$ must contain at least $C_G|E|$ distinct copies of any given weighted tree $G$, where $C_G$ is a constant depending only on the graph $G$.
\end{abstract}

\section{Introduction}
Many questions in combinatorics involve the behavior of the distance set $\Delta(E)$ of a set $E$ , defined as $\Delta(E) = \{ d(\mathbf{x}, \mathbf{y}) : \mathbf{x}, \mathbf{y} \in E\}$ for some distance function $d$. For instance, Erd\H{o}s' celebrated distinct distances problem conjectured that for finite sets $E \subset \arr^2$,  $|\Delta(E)| \gtrsim |E|^{1 - \epsilon}$ for any positive $\epsilon$. This conjecture was proven in this form by Guth and Katz in \cite{GK15}. Distance problems where the ambient space is a finite vector space have also been a subject of much research \cite{BKT04, IR07, KS12, V08}.

A natural question follows: under what conditions on $E$ can we find not just pairs of points a specified distance apart, but groups of points with some specfied pattern of distances? In other words, given some weighted graph $G$ where the edge weights correspond to distances between points, when can we find a ``copy" of $G$ inside $E$? Bennett, Chapman, Covert, Hart, Iosevich, and Pakianathan proved a result in this direction for the Euclidean distance in $\zee_p^d$ and path graphs in \cite{BCCHIP}, and McDonald gives a similar result for dot products in $\zee_p^d$ in \cite{M16}. In this paper, we give an answer to this question for wide classes of graphs, distances, and ambient sets.  Moreover, we show that these questions, and other similar ones, are part of a much more general framework. 

\begin{defn}
\label{ksurj}
 Given an ambient set $X$ and some set $D$ of possible distances, a \textbf{symmetric distance function} is a function $d: X \times X \to D$ such that $d(x_1, x_2) = d(x_2, x_1)$ for all $x_1, x_2 \in X$. Such a function is $\mathbf{K}$\textbf{-surjective} if for every $E \subset X$ with $|E| \geq K$, the restriction of $d$ to $E \times E$ is surjective. In other words, every set of size at least $E$ determines every distance.
\end{defn}

Some motivation for this definition is provided by the following theorem, which will allow us to apply the results of this paper to a common setting for distance problems. Throughout, we will use $\eff_q$ to refer to the unique finite field with $q$ elements for some prime power $q$, and we will let $\eff_q^d$ be a $d$-dimensional vector space over this field.
\begin{theorem}[A. Iosevich and M. Rudnev (2007) \cite{IR07}]
\label{euclid}
Let $X =  \eff_q^d$, $D = \eff_q$, and define $d(\{ x_i\}, \{ y_i\}) = \sum(x_i - y_i)^2$. Then $d$ is $K$-surjective with $K = Cq^{(d + 1) / 2}$ for some constant $C$ independent of $q$.
\end{theorem}

Now we just need to define a graph embedding, which is done in the natural way:



\begin{defn}
\label{embed}
Suppose we have a space $X$, a set of distances $D$, and a symmetric distance function $d$. Then for a weighted graph $G$ with edge weights in $D$, an \textbf{embedding} of $G$ into $X$ is an injective function $f: V(G) \to X$ such that for every edge $(v_1, v_2) \in E(G)$ with weight $t$,
\[
d(f(v_1), f(v_2)) = t
\]
We will typically identify such an embedding with its image in $X$. A collection of such embeddings $\{ f_i\}_{i \in I}$ is \textbf{disjoint} if all its images are disjoint subsets of $X$.
\end{defn}

\subsection{Results}

The constant in our result will depend on an invariant of the graph $G$:
\begin{defn}
\label{stringiness}
Let $G$ be a finite nonempty connected graph; let the degrees of the vertices of $G$ be $d_1, d_2, \dots, d_n$, ordered so that $d_1 \geq d_2 \geq \dots \geq d_n$. Then the $stringiness$ of $G$, denoted $\sigma(G)$, is defined to be 
\[
\sigma(G) = (d_1 + 1)\prod_{i = 2}^{n} d_i .
\]
\end{defn} 
For example, the stringiness of the Petersen graph is $4 \cdot 3^9 = 78732$. The following estimate gives bounds on the stringiness of a tree:

\begin{theorem}
\label{stringinessbounds}
Let $G$ be a nonempty tree with $n$ edges. Then
\[
n + 1 \leq \sigma(G) \leq 2^n.
\]
\end{theorem}
The lower bound is sharp (the star graph $K_{1, n}$ attains it) while the upper bound is sharp up to a constant (the path graph $P_{n + 1}$ has stringiness $3 \cdot 2^{n - 2}$).

Now we are ready to state our main theorem:
\begin{theorem}[Main theorem]
\label{main}
Let $X$ be a set with a symmetric distance function $d$ to a set of distances $D$, let $d$ be $K$-surjective, and let $E \subseteq X$ with $|E| = rK$. Then for any weighted tree $G$ with edge weights in $D$, there exists a disjoint collection $A_G$ of embeddings of $G$ into $E$ with
\[
|A_G| \geq \left(\frac{r}{\sigma(G)} - 1\right) K
\]
\end{theorem}
\begin{corollary}
If $|E| \geq 2\sigma(G)K$, then there  is at least one embedding of $G$ into $E$.
\end{corollary}

Combining this result with Theorem \ref{euclid} gives the following result in $\eff_q^d$:

\begin{corollary}
\label{maincorr}
Let $G$ be a weighted tree with edge weights in $\eff_q$. Then there exists a constant $C$ independent of $q$ and $G$ such that every subset $E$ of $\eff_q^d$ such that $|E| \geq C\sigma(G)q^{\frac{d+ 1}{2}}$ contains an embedding of $G$.
\end{corollary}

We compare this result to the following previous result:
\begin{theorem}[Bennett, Chapman, Covert, Hart, Iosevich, Pakianathan, \cite{BCCHIP}]
\label{knownpaths}
Let $G$ be a weighted path or star graph with edge weights in $\eff_q$, and suppose $G$ has $k$ edges. Then there exists a constant $C$ independent of $q$ and $G$ such that every subset $E$ of $\eff_q^d$ such that $|E| \geq Ckq^{\frac{d+ 1}{2}}$ contains an embedding of $G$.
\end{theorem}

The idea of this paper is to show that results of the type of Theorem \ref{knownpaths} are instances of a more general phenomenon. Corollary \ref{maincorr} is an example of this phenomenon applying to the standard distance on $\eff_q^d$; it allows extending these results to general trees for which no results existed. We also note that proof of Theorem \ref{main} is very modular, so it is possible to (for instance) use Theorem \ref{knownpaths} to get slightly better bounds for embeddings of more complicated trees in $\zee_q^d$. 

We proceed as follows: for illustrative purposes, we first state and prove Theorem \ref{easy}, a weaker version of Theorem \ref{main}, using Lemma \ref{easylemma}. We then give the very similar proof of Theorem \ref{main}, which relies on the analagous Lemma \ref{main2}. Finally we prove Theorem \ref{stringinessbounds}.

\section{Graph embeddings}

\subsection{An easier case of \ref{main}}
The proof of \ref{easy} is by induction; it is convenient to state the base case as a separate lemma.
\begin{lemma}
Let $X$ be a set with a symmetric distance function $d$ to a set of distances $D$, let $d$ be $K$-surjective, and let $E \subseteq X$ with $|E| = rK$. Then for any fixed $t \in D$, there are at least $\frac{r - 1}{2}K$ disjoint pairs $\{ e_i, f_i\}_{i \in I}$ in $E$ such that $d(e_i, f_i) = t$ for all $i \in I$.
\label{easylemma}
\end{lemma}
\begin{proof}
Since $E$ is finite, let $I$ be maximal. Let $F$ be the union of all the pairs. Then by maximality of $I$, $E - F$ cannot contain any pair of points with distance $t$. By $K$-surjectivity, this means $|E - F| < K$; so $|F| \geq (r - 1)K$, meaning $|I| = |F| / 2 \geq \frac{r - 1}{2}K$ as required.
\end{proof}

\begin{theorem}
\label{easy}
Let $X$ be a set with a symmetric distance function $d$ to a set of distances $D$, let $d$ be $K$-surjective, and let $E \subseteq X$ with $|E| = rK$. Then for any weighted nonempty tree $G$ with edge weights in $D$, suppose $G$ has $n$ edges. Then there exists a disjoint collection $A_G$ of embeddings of $G$ into $E$ with
\[
|A_G| \geq \left(\frac{r + 1}{2^n} - 1\right) K.
\]
\end{theorem}
\begin{proof}
We proceed by induction on $n$. The case $n = 0$ is tautological, and the case $n = 1$ is equivalent to Lemma \ref{easylemma}. 

So assume $n \geq 2$, which means $G$ must have a leaf. Fix any leaf; call it $v$, its associated edge $e$, the unique vertex adjacent to it $w$, and the edge weight $t$. Then $G - v$ is a tree with strictly fewer edges; so by the inductive hypothesis we have a disjoint collection $A_{G - v}$ of embeddings of $G - v$ into $E$ with 
\[
|A_{G - v}| \geq \left(\frac{r + 1}{2^{n - 1}} - 1\right) K.
\]

Consider the set $W = \{ f(w)| f \in A_{G - v}\}$. By Lemma \ref{easylemma}, there exist at least $\frac{|W| / K - 1}{2}$ disjoint pairs of points $(\{f(w), f'(w) \}$ in $W$ with $d(f(w), f'(w))  = t$. But for each of these pairs we can consider the function $g: V(G) \to E$ given by 
\[
g(x) = 
\begin{cases}
f'(w) & x = v \\
f(x) & x \neq v
\end{cases}
\]
By construction, these are disjoint embeddings of $G$ into $E$, and there are at least $\frac{|W| / K - 1}{2}K$ of them; but by disjointness of $A_{G - v}$, $|W| = |A_{G - v}|$ so there are at least
\[
\frac{|A_{G - v}| / K - 1}{2}K \geq \frac{\left(\frac{r + 1}{2^{n - 1}} - 1\right) - 1}{2} K = \left(\frac{r + 1}{2^n} - 1\right) K
\]
disjoint embeddings of $G$ as required.

\end{proof}
\begin{corollary}
If $|E| \geq 2^{n + 1}K$, there is at least one embedding of $G$ into $E$.
\end{corollary}

Note that this proof would have worked equally well even if $d$ were not symmetric, which will not carry over to Theorem \ref{main}.

\subsection{Proof of Theorem \ref{main}}
Analagously to the proof of Theorem \ref{easy}, we will state a governing lemma (Lemma \ref{main2}); the structure will be identical except that we are building our graph $G$ out of stars instead of working purely with edges. For technical reasons, the application of the $K$-surjectivity assumption is more difficult in this case, so we will first state an auxilliary lemma, Lemma \ref{main1}.
\begin{lemma}
\label{main1}
Let $X$ be a set with a symmetric distance function $d$ to a set of distances $D$, let $d$ be $K$-surjective, and let $E \subseteq X$ with $|E| = rK$. Then for any fixed $t \in D$, $s \in \enn$, there are at most $sK$ points $e \in E$ for which there are not at least $s$ other distinct points $e_1, e_2, \dots, e_s \in E$ such that $d(e, e_i) = t$ for all $i$.
\end{lemma}
\begin{proof}
Create a graph $H$ whose vertices are the points of $E$ and where two vertices are connected by an edge if and only if the corresponding points of $E$ are distance $t$. Then consider the subgraph $H^*$ of $H$ generated by only those vertices of degree $< s$. By construction, the maximum degree of vertices in $H$ is less than $s$, which means $H^*$ can be $s$-colored, that is, partitioned into $s$ independent sets. Since the $K$-surjectivity condition guarantees that an independent set in $H$ (and thus in $H^*$) has size $< K$, it follows that $|H^*| < sK$. This means that at most $sK$ of the vertices of $H$ have degree $<s$, proving the lemma.
\end{proof}

\begin{lemma}
\label{main2}
Let $X$ be a set with a symmetric distance function $d$ to a set of distances $D$, let $d$ be $K$-surjective, and let $E \subseteq X$ with $|E| = rK$. Then for any weighted nonempty star graph $G \iso K_{1, n}$ with edge weights in $D$, $E$ contains at least $\frac{r - n}{n + 1}K$ disjoint embeddings of $G$.
\end{lemma}

\begin{proof}
As in the proof of Lemma \ref{easylemma}, let $I$ be a maximal set of embeddings of $G$ into $E$, and define $F$ to be the union of all the embeddings contained in $I$. Then $E - F$ must have no embeddings of $G$.

Suppose the set of edge weights of $G$ is $\{ t_1, t_2, \dots, t_a\}$  appearing with multiplicities $\{m_1, m_2, \dots m_a\}$ respectively. Then by Lemma \ref{main1}, for each fixed $i$ there are at most $m_iK$ points of $E - F$ which are not distance $t_i$ from at least $m_i$ other points of $E - F$. Summing over $i$ we get that there are at most
\[
\sum_{i = 1}^{a} m_iK = nK
\]
points of $E - F$ which are not distance $t_i$ from at least $m_i$ other points of $E - F$ for every $i$. But if $x \in E - F$ is in fact distance  $t_i$ from at least $m_i$ other points of $E - F$ for every $i$, then $x$ forms the root of an embedding of the graph $G$. Thus $|E  - F| \leq nK$; so $|I| = |F|/(n + 1) \geq \frac{rK - nK}{n + 1} = \frac{r - n}{n + 1}K$ as required.
\end{proof}

Note that when $n = 1$, this is exactly the statement of Lemma \ref{easylemma}, but when $n \geq 2$ we may have to deal with repeated edge weights, which adds the extra complexity.

Now we are ready to prove the main theorem, Theorem \ref{main}:

\begin{proof}
The proof proceeds by strong induction on the number of edges in $G$. If $G$ contains no edges, $\sigma(G) = 1$ and the theorem is clearly true; if $G$ is a star graph $K_{1, n}$, then $\sigma(G) = n + 1$ and the theorem is just Lemma \ref{main2}. 

So assume $G$ is not a star graph. Then we can always find a vertex $w \in G$ such that:
\begin{enumerate}
\item $w$ is not a leaf of $G$ (equivalently, $\deg_G w \geq 2$).
\item All but one of the vertices connected to $G$ are leaves (call these leaves $v_1, v_2, \dots, v_y$, and the associated distances $t_1, t_2, \dots, t_y$, where $y = \deg_G w - 1$.)
\item There exists a vertex $v \neq w$ in $G$ such that $\deg_G v \geq \deg_G w$.
\end{enumerate}
To see this, let $L$ be the set of leaves of $G$. Then since $G$ is a tree which is not a star graph, $G - L$ is a tree which contains at least one edge; any leaf of $G - L$ satisfies conditions 1 and 2, and since $G - L$ has at least two leaves, at least one of these must satisfy condition 3.

Define the graph $H$ to be $G - \{v_1, v_2, \dots v_y\}$. By conditions 1 and 2, $H$ is a tree with fewer edges than $G$; by condition 3, $\sigma(H) = \frac{\sigma(G)}{y + 1}$ (since we can reorder the product in Definition \ref{stringiness} to make $v$ correspond to $d_1$, and all we lose by deleting these leaves is a factor of $\deg_G w$). By the inductive hypothesis we have a disjoint collection $A_H$ of embeddings of $H$ into $E$ with 
\[
|A_H| \geq \left(\frac{r}{\sigma(H)} - 1\right) K.
\]

Consider the set $W = \{ f(w)| f \in A_H\}$. By Lemma \ref{main2}, there exist at least $\frac{|W| / K - n}{n + 1}$ disjoint sets of points $\{f(w), f_1(w), \dots, f_y(w) \}$ contained in $W$ with $d(f(w), f_i(w)) = t_i$ for every $i$ (i.e. embedddings of a particular star graph). But for each of these sets we can consider the function $g: V(G) \to E$ given by 
\[
g(x) = 
\begin{cases}
f_i(w) & x = v_i \\
f(x) & x \notin \{v_1 \dots v_y \}
\end{cases}
\]
By construction, these are disjoint embeddings of $G$ into $E$, and there are at least $\frac{|W| / K - n}{n + 1}K$ of them; but by disjointness of $A_H$, $|W| = |A_H|$ so there are at least
\begin{align*}
\frac{|A_H| / K - n}{n + 1}K &\geq \frac{\left(\frac{r}{\sigma(H)} - 1\right) - n}{n + 1} K \\
&= \left(\frac{r}{(n + 1) \cdot \sigma(H)} - 1\right) K \\
 &= \left(\frac{r}{\sigma(G)} - 1\right) K.
\end{align*}
disjoint embeddings of $G$ as required.

\end{proof}
Note that in view of Theorem \ref{stringinessbounds}, this is stronger than Theorem \ref{easy}

\subsection{Proof of Theorem \ref{stringinessbounds}}
\begin{proof}
For the lower bound on $\sigma(n)$, a simple inductive argument suffices; if $n = 1$ then $\sigma(n) = 2$ since there is only one possible graph with one edge; if $n > 1$ then deleting a leaf must decrease the stringiness since the degree of the other vertex decreases; so for a graph with $n$ edges, $\sigma(G) > \sigma(G - v) \geq n$; so $\sigma(G) \geq n + 1$ (since $\sigma(G) \in \zee$).

For the upper bound, write the degrees of the vertices of $G$ as $d_1, d_2, \dots d_{n +1}$; without loss of generality say $d_{n + 1} = 1$. Then by the arithmetic-geometric mean inequality:
\begin{align*}
\left((d_1 + 1)\cdot d_2 \cdot d_3 \cdot \dots \cdot d_n\right)^{1/n} &\leq \frac{d_1 + 1 + d_2 + d_3 + \dots d_n}{n} \\
\sigma(G)^{1 /n} & \leq \frac{\sum d_i}{n} \\
\sigma(G)^{1 /n} & \leq \frac{2n}{n} \\
\sigma(G) & \leq 2^n.
\end{align*}
\end{proof}

\section{Acknowledgements}
I would like to thank Alex Iosevich and Steven Senger for their help, support and feedback.

\end{document}